\documentclass{article}

\usepackage{amsfonts, mathtools, bm, amsthm}
\usepackage[hidelinks]{hyperref}
\usepackage{cleveref}
\usepackage[inline]{enumitem}
\usepackage[top=1in, bottom=1.25in, left=1.25in, right=1.25in]{geometry}
\usepackage{natbib}

\newtheorem{theorem}{Theorem}
\newtheorem{corollary}{Corollary}
\newtheorem{proposition}{Proposition}
\newtheorem{lemma}{Lemma}

\theoremstyle{definition}
\newtheorem{definition}{Definition}
\newtheorem{example}{Example}

\theoremstyle{remark}
\newtheorem{remark}{Remark}

\def\eqns#1{\begin{equation*}#1\end{equation*}}
\def\eqnsml#1{\begin{multline*}#1\end{multline*}}
\def\eqnl#1#2{\begin{equation}\label{#1}#2\end{equation}}
\def\eqnsa#1{\begin{subequations}\begin{align*}#1\end{align*}\end{subequations}}
\def\eqnla#1#2{\begin{subequations}\label{#1}\begin{align}#2\end{align}\end{subequations}}

\def\one{\mathbf{1}}

\def\calB{\mathcal{B}}
\def\calI{\mathcal{I}}
\def\calG{\mathcal{G}}
\def\bscalI{\bm{\calI}}

\def\bbE{\mathbb{E}}
\def\bbI{\mathbb{I}}
\def\bbN{\mathbb{N}}
\def\bbP{\mathbb{P}}
\def\bbR{\mathbb{R}}

\def\d{\mathrm{d}}

\DeclareMathOperator{\Sym}{Sym}

\def\defeq{\doteq}
\def\given{\mid}
\def\Given{\,\bigg|\,}
\def\ind#1{\one_{#1}}

\begin{document}

\title{On the loss of Fisher information in some multi-object tracking observation models}

\author{J. Houssineau\thanks{Department of Statistics and Applied Probability, National University of Singapore. Email:\href{mailto:stahje@nus.edu.sg}{stahje@nus.edu.sg}},
\quad A. Jasra\thanks{Department of Statistics and Applied Probability, National University of Singapore Email:\href{mailto:staja@nus.edu.sg}{staja@nus.edu.sg}}
\quad and \quad S. S. Singh\thanks{Department of Engineering, University of Cambridge and the Alan Turing Institute. Email:\href{mailto:sss40@cam.ac.uk}{sss40@cam.ac.uk}}%
}

\date{}

\maketitle

\begin{abstract}
The concept of Fisher information can be useful even in cases where the probability distributions of interest are not absolutely continuous with respect to the natural reference measure on the underlying space. Practical examples where this extension is useful are provided in the context of multi-object tracking statistical models. Upon defining the Fisher information without introducing a reference measure, we provide remarkably concise proofs of the loss of Fisher information in some widely used multi-object tracking observation models.
\end{abstract}

\section{Introduction}

The Fisher information is a fundamental concept in Statistics and Information Theory \citep{Rissanen1996}, e.g.\ it features in Jeffreys prior \citep{Jeffreys1946}, the Cram\'er-Rao lower bound \citep{Cramer1946, Rao1992} and in the analysis of the asymptotics of maximum-likelihood estimators \citep{LeCam1986, Douc2004, Douc2011}. Although different generalisations have been proposed, see e.g.\ \citep{Lutwak2005,Lutwak2012}, the standard formulation of the Fisher information often involves a parametric family of probability measures which are all absolutely continuous with respect to a common reference measure in order to define the corresponding probability density functions. This though can be a restrictive assumption for some statistical models.

Let $\Theta \subseteq \bbR$ be a given open set of parameters and let $\{P_{\theta}\}_{\theta \in \Theta}$ be a parametric family of probability measures on a Polish space $E$ equipped with its Borel $\sigma$-algebra $\calB(E)$ and with a reference measure~$\lambda$. Most often, $E$ is a subset of $\bbR^d$ for some $d > 0$ and $\lambda$ is the Lebesgue measure, although Haar measures can be considered more generally for locally-compact topological groups. We will consider the former since the main practical limitation with the usual definition of Fisher information does not come from the lack of natural reference measure but instead from the irregularity of the probability distributions of interest. The usual setting is to assume that for all $\theta \in \Theta$ it holds that $P_{\theta}$ is absolutely continuous with respect to $\lambda$, denoted $P_{\theta} \ll \lambda$. In this case, the probability density function $p_{\theta}$ can be defined as the Radon-Nikodym derivative
\eqns{
p_{\theta} = \dfrac{\d P_{\theta}}{\d \lambda}
}
that is, as the function on $E$ defined uniquely up to a $\lambda$-null set by
\eqns{
P_{\theta}(A) = \int \ind{A}(x) p_{\theta}(x) \lambda(\d x)
}
for all $A \in \calB(E)$. In this situation, assuming that $p_{\theta}$ is differentiable with respect to $\theta$, the \emph{score} is defined as $\frac{\partial}{\partial \theta} \log p_{\theta}(x)$ or indeed ${ \frac{\partial}{\partial \theta} p_{\theta}(x)}/{ p_{\theta}(x) }$. Under the final assumption that the score is square integrable,
the Fisher information \citep{Lehmann1998} is defined as
\eqnl{eq:stdFisher}{
\calI(\theta) = \int \bigg( \frac{ \frac{\partial}{\partial \theta} p_{\theta}(x)}{ p_{\theta}(x) } \bigg)^2 p_{\theta}(x) \lambda(\d x).
}
The objective in this article is twofold.  For some applications, it is necessary to relax the requirement  that $P_{\theta} \ll \lambda$ holds for all $\theta \in \Theta$, or indeed any $\theta$, and an appropriate definition of $\calI(\theta)$ is needed in these cases. Upon addressing this issue, our second objective is then to study the Fisher information of some observation models frequently used in multi-object tracking. Our starting point is the following generalisation of the score ${ \frac{\partial}{\partial \theta} p_{\theta}(x)}/{ p_{\theta}(x) }$ given in \cite{Heidergott2008},
\eqnl{eq:score}{
\dfrac{ \d P'_{\theta}}{\d P_{\theta} }(x),
}
where $P'_{\theta}$ is the (yet to be formally defined) derivative of the probability measure $P_{\theta}$ with respect to $\theta$ and the ratio in \cref{eq:score} is the Radon-Nikodym derivative of $P'_{\theta}$ with respect to $P_{\theta}$. 
\cite{Heidergott2008} introduced this definition of the  score in the context of sensitivity analysis for performance measures of Markov chains \citep{Rubinstein1993}. We define the Fisher information using this expression for the score and then study the loss of information in the context of some statistical estimation problems arising in Engineering (see \cref{sec:practicalExample}.) Indeed, as shown in \cref{res:consistency}, when the family $P_{\theta}$  have differentiable densities with respect to the Lebesgue measure, the Fisher information defined using the score in \cref{eq:score} coincides with  \cref{eq:stdFisher}. 

The first problem studied in \cref{sec:practicalExample:randomPermutation} concerns fitting a parametric model to random vectors which are observed through a sensor that randomly permutes the components of the vector. This problem arises in the context of multi-object tracking \citep{Houssineau2017_identification} where the random vector corresponds to recorded measurements from distinct objects (e.g.\ vehicles) being tracked using a radar. The radar is able to provide (noisy) measurements of the locations of these object but without knowledge of the association of recorded measurements to the objects themselves. Our analysis involves studying a parametric model that does not have a common dominating measure and through the proposed definition of the Fisher information we provide a simple proof that association uncertainty results in a loss of information.  This fact is surprisingly undocumented in the literature despite the numerous articles in Engineering on statistical inference for these types of models.

Multi-object observation models often also include thinning and clutter. Clutter are spurious observations, unrelated to the objects being tracked, generated by radar reflections from non-targets. Thinning is the random deletion of target generated measurements which models the occasional obscuring of targets by obstacles. The augmented set of thinned and spurious observations can be modelled as a spatial point process and \cref{sec:practicalExample:pointProcess} concerns fitting a parametric model to a spatial point process that is observed under thinning and superposition. Like random permutation, thinning and superposition results in a loss of information, which is easily shown using the Fisher information defined via \cref{eq:score} and its associated properties. These properties are invoked in the proofs in \cref{sec:practicalExample} but are formally stated and proven in the final section, \cref{sec:measTheroreticFisher}.

\section{Motivating examples}
\label{sec:practicalExample}

\subsection{Random permutation of a random vector}
\label{sec:practicalExample:randomPermutation}

Consider a parametric probability measure $P_{\theta}$, $\theta \in \Theta \subseteq \bbR$. For each $\theta$, $P_{\theta}$ is the 
law of a random vector $(X_1,\dots,X_n)$ where each $X_i$ are in $\bbR^d$, i.e. $P_{\theta}$ is a probability measure on $(\bbR^{dn},\calB(\bbR^{dn}))$.  Assume $n,d \in \bbN$ are fixed. Let $(X'_1,\dots,X'_n) = (X_{\varsigma(1)},\dots,X_{\varsigma(n)})$, a random permutation of $(X_1,\dots,X_n)$, where $\varsigma$ is a random variable with values in the set $\Sym(n)$ of permutations of $\{1,\dots,n\}$. Throughout this section, $x_{1:n}$ denotes the vector $(x_1,\dots,x_n)$.

In multi-object tracking, each $X_i$ corresponds to a measurement of a distinct object being tracked; there are $n$ of them. The sensor acquiring $(X_1,\dots,X_n)$, e.g.\ a radar, returns the vector but with the association of observations to the $n$ targets lost, which can be modelled as $(X'_1,\dots,X'_n)$. Filtering for such models has spawned an entire family of algorithms. e.g.\ see \cite{Blackman1986,BarShalom1987}.

The following \lcnamecref{res:randomPerturbation} shows that the Fisher information $\bscalI'(\theta)$ of the law of $X'_{1:n}$, i.e.\ after the random permutation, is smaller than the Fisher information $\bscalI(\theta)$ of $P_{\theta}$. The concept of weak-differentiability will be defined formally in the next section.

\begin{theorem}
\label{res:randomPerturbation}
Assume the family $\{P_{\theta}\}_{\theta \in \Theta}$ is weakly-differentiable.  Then any random permutation of $X_{1:n}$ that is independent of $\theta$ incurs a loss of information, that is $\bscalI'(\theta) \leq \bscalI(\theta)$.
\end{theorem}

\begin{proof}
Let $\pi$ be the probability distribution of $\varsigma$ on $\Sym(n)$, then a version of the conditional law of $X'_{1:n}$ given $X_{1:n}$ is
\eqns{
Q(B_1 \times \dots \times B_n \given X_{1:n}) = \sum_{\sigma \in \Sym(n)} \pi(\sigma) \prod_{i = 1}^n \delta_{X_{\sigma(i)}}(B_i),
}
for any $B_1 \times \dots \times B_n \in \calB(\bbR^{dn})$. The fact that $Q$ does not depend on $\theta$ follows from the independence of the random permutation $\varsigma$ from the parameter. From \cref{res:jointRadonNikodym,res:derivativeProduct} the score corresponding to the extended model $(X_{1:n},X'_{1:n})$ can then be expressed as
\eqnla{eq:proof:randomPerturbation}{
\dfrac{\d (P_{\theta} \times Q)'}{\d P_{\theta} \times Q}(x_{1:n},x'_{1:n}) & = \dfrac{\d P'_{\theta} \times Q}{\d P_{\theta} \times Q}(x_{1:n},x'_{1:n}) \\
& = \dfrac{\d P'_{\theta}}{\d P_{\theta}}(x_{1:n})
}
for all $x_{1:n}$ and all $x'_{1:n}$ in $\bbR^{dn}$. Note that $(X_{1:n},X'_{1:n})$ is not absolutely continuous with respect to the Lebesgue measure on $\bbR^{2dn}$ even when $P_{\theta}$ has a density with respect to the Lebesgue measure. Using the extension of the Fisher identity (see \cref{res:FisherIdentity}), it follows that
\eqnla{eq:proof:randomPerturbation:FisherIdentity}{
\dfrac{\d\hat{P}'_{\theta}}{\d \hat{P}_{\theta}}(X'_{1:n}) & = \bbE_{\theta}\bigg( \dfrac{\d (P_{\theta} \times Q)'}{\d P_{\theta} \times Q}(X_{1:n},X'_{1:n}) \Given X'_{1:n} \bigg)\\
& = \bbE_{\theta}\bigg( \dfrac{\d P'_{\theta}}{\d P_{\theta}}(X_{1:n}) \Given X'_{1:n} \bigg) \quad \text{almost surely},
}
with $\hat{P}_{\theta}$ the marginal law of $X'_{1:n}$. Applying Jensen's inequality to the function $y \mapsto y^2$, we conclude that
\eqns{
\bscalI'(\theta) \defeq \bbE_{\theta}\Bigg( \bigg( \dfrac{\d \hat{P}'_{\theta}}{\d \hat{P}_{\theta}}(X'_{1:n})\bigg)^2 \Bigg) \leq  \bbE_{\theta}\Bigg( \bigg( \dfrac{\d P'_{\theta}}{\d P_{\theta}}(X_{1:n})\bigg)^2 \Bigg) = \bscalI(\theta),
}
which concludes the proof of the \lcnamecref{res:randomPerturbation}.
\end{proof}

\begin{remark}
A different proof of this result has been proposed in \cite{Houssineau2017_identification} using the standard formulation of Fisher information. However the proof presented here is remarkably concise and less tedious thanks to the possibility of defining in \cref{eq:proof:randomPerturbation} the score of the extended parametric model $(X_{1:n},X'_{1:n})$ which does not have a common dominating measure. The final result then follows from the identity in \cref{eq:proof:randomPerturbation:FisherIdentity} and Jensen's inequality. 
\end{remark}
It is not possible to establish a strict information loss in general, e.g.\ if $P_{\theta}$ is symmetrical or if $\theta$ is related to some summary statistics that is not affected by random permutation. Additional assumption that guarantee a strict loss are given in \cite{Houssineau2017_identification}.

\subsection{Thinning and superposition of point processes}
\label{sec:practicalExample:pointProcess}

Spatial point processes are important in numerous applications \citep{Baddeley2006}, e.g.\ Forestry \citep{Stoyan2000} and Epidemiology \citep{Elliot2000}. In addition, point process models are widely used in formulating multi-object tracking problems \citep{Mahler2007} as they  naturally account for an unknown number of objects which are observed indirectly without association and under thinning and superposition.   We adopt the approach of the previous section but now characterise the Fisher information of a family of point process parametrized by $\theta \in \Theta$ observed under thinning and superposition. (Note the loss of Fisher information in the presence of association uncertainty has already been established in \cref{sec:practicalExample:randomPermutation}.) 

Let $\Phi$ denote a point process on $\bbR^d$ with parametrised distribution $P_{\theta}$ on $E = \bigcup_{n \geq 0} \bbR^{dn}$, with $\bbR^0$ denotes an arbitrary isolated point representing the absence of points in the process. A realisation from $P_{\theta}$ is a random vector $(x_1,\ldots,x_n)$ where both the number of points $n$ and their locations $x_i \in \bbR^{d}$ are random. However, point-process distributions on $\bbR^d$ are not always absolutely continuous with respect to the corresponding Lebesgue measure. In particular, the distribution of a non-simple point process, which is a point process such that there is a positive probability of two or more points of its realisation, say $x_i$ and $x_j$ of $(x_1,\ldots,x_n)$,  being identical; see \cite{Schoenberg2006} for a discussion about non-simple point processes and examples, e.g.\ by duplicating the points in a realisation as discussed further below. Assuming that the family $\{P_{\theta}\}_{\theta \in \Theta}$ is weakly-differentiable, the Fisher information $\bscalI_{\Phi}(\theta)$ corresponding to the parametrised distribution of $\Phi$ can then be expressed as
\eqnl{eq:PhiFisher}{
\bscalI_{\Phi}(\theta) = \sum_{n \geq 0} \pi_{\theta}(n) \int \bigg( \dfrac{\d P_{\theta}'}{\d P_{\theta}}(x_1,\dots,x_n) \bigg)^2 P_{\theta}(\d(x_1,\dots,x_n) \given n),
}
where $\pi_{\theta}$ is a probability mass function on $\bbN_0$ characterising the number of points $N$ in $\Phi$ and where $P_{\theta}(\cdot \given n)$ is the conditional distribution of the location of the points in $\Phi$ given that the number of points is $n$ (which is supported by $\bbR^{dn}$). A straightforward example is when $\Phi$ is an independently identically distributed point process. Its distribution factorises as
\eqns{
P_{\theta}(B_1 \times \dots \times B_n) = \pi_{\theta}(n) \prod_{i=1}^n \mu_{\theta}(B_i)
}
for any $B_1,\dots,B_n \in \calB(\bbR^d)$ and any $n \in \bbN_0$, where $\mu_{\theta}$ is a probability measure on $\bbR^d$. Using the product rule of \cref{res:derivativeProduct} the expression of the Fisher information simplifies in the independently identically distributed case to
\begin{align}\label{eq:iidppFisher}
\bscalI_{\Phi}(\theta) & = \calI_N(\theta) + \sum_{n \geq 0} n^2 \pi_{\theta}(n) \int \bigg( \dfrac{\d \mu_{\theta}'}{\d \mu_{\theta}}(x) \bigg)^2 \mu_{\theta}(\d x) \nonumber \\
& = \calI_N(\theta) + \bbE(N^2) \bscalI_X(\theta)
\end{align}
where $X$ is a random variables with distribution $\mu_{\theta}$.

\begin{example}
A trivial construction of a non-simple point process can be obtained from an independently identically distributed point process $\Phi$ by duplicating its realisation. The resulting point process, denoted $\Phi_2$, has each point of $\Phi$ present twice. The Fisher information of $\Phi_2$ can be expressed with the proposed formulation in spite of the lack of absolute continuity with respect to to the reference measure on~$E$. Indeed, the law $P^+_{\theta}$ of the point process $\Phi_2$ is
\eqns{
P^+_{\theta}(B_1\times\dots\times B_{2n}) = \pi_{\theta}(n) \sum_{\sigma \in \Sym(2n)}\prod_{i=1}^n \bar\mu_{\theta}(B_{\sigma(2i-1)} \times B_{\sigma(2i)})
}
and $P^+_{\theta}(\bbR^{d(2n+1)}) = 0$, where $\bar\mu_{\theta}$ a probability measure supported by the diagonal of $\bbR^d \times \bbR^d$ such that $\bar\mu_{\theta}(B \times B') = \mu_{\theta}(B \cap B')$ for any $B,B' \in \calB(\bbR^d)$. One can verify that $\bar\mu_{\theta}'(B \times B') = \mu_{\theta}'(B \cap B')$ so that
\eqns{
\dfrac{\d\bar\mu_{\theta}'}{\d\bar\mu_{\theta}}(x,x') =
\begin{cases*}
\dfrac{\d \mu_{\theta}'}{\d \mu_{\theta}}(x) & if $x = x'$ \\
0 & otherwise,
\end{cases*}
}
from which it follows that $\bscalI_{\Phi_2}(\theta) = \bscalI_{\Phi}(\theta)$, that is, duplicating each point in the point process $\Phi$ does not change the Fisher information. In the context of parameter inference, this is in agreement with the natural approach of removing the duplicate points before estimating $\theta$.
\end{example}

Returning now to a general point process $\Phi$ which is not necessarily independently identically distributed. For each $\alpha \in [0,1]$, let $\Phi_{\alpha}$ denote the thinned version of~$\Phi$ where each point of its realisation is retained independently of the other points with probability $\alpha$. In multi-object tracking, an independently thinned point processes arises because a radar can fail to return a credible observation for an object in its surveillance region. 

\begin{theorem}
\label{res:thinning}
Let $\Phi$ be a point process characterised by a weakly-differentiable family of probability distributions parametrised by $\Theta$, then $\bscalI_{\Phi}(\theta) \geq \bscalI_{\Phi_{\alpha}}(\theta)$ holds for any $\alpha \in [0,1]$. If $\bscalI_{\Phi}(\theta)>0$ then the inequality is strict when $\alpha<1$.
\end{theorem}

\begin{proof}
The probability distribution $Q_{\alpha}$ of the thinned point process $\Phi_{\alpha}$ given $\Phi$ can be expressed as
\eqns{
Q_{\alpha}(B_1 \times \dots \times B_k \given x_1,\dots,x_n) =
\sum_{I \subseteq \{1,\dots,n\} : |I| = k} \alpha^k(1-\alpha)^{n-k} \prod_{i \in I} \delta_{x_i}  \big(B_{s(i)}\big)
}
for any $B_1,\dots,B_k \in \calB(\bbR^d)$, any $x_1,\dots,x_n \in \bbR^d$ and any integers $n,k$ such that $k \leq n$, with $s(i) = |\{1,\dots,i\} \cap I|$ so that $i$ is the $s(i)$\textsuperscript{th} element of~$I$. We obtain from the Fisher identity that the score associated with the point process $\Phi_{\alpha}$ with law $P_{\theta,\alpha}$ verifies
\eqnsa{
\dfrac{\d P'_{\theta,\alpha}}{\d P_{\theta,\alpha}}(x_1,\dots,x_k) & =  \bbE\bigg( \dfrac{\d P'_{\theta} \times Q_{\alpha}}{\d P_{\theta} \times Q_{\alpha}}(\Phi,\Phi_{\alpha}) \Given \Phi_{\alpha} = (x_1,\dots,x_k)  \bigg) \\
& = \bbE\bigg( \dfrac{\d P'_{\theta}}{\d P_{\theta}}(\Phi) \Given \Phi_{\alpha} = (x_1,\dots,x_k)  \bigg),
}
where the use of $\Phi$ as an argument of point-process distributions is possible because of the irrelevance of the points' ordering. The proof of $\bscalI_{\Phi}(\theta) \geq \bscalI_{\Phi_{\alpha}}(\theta)$ can now be concluded using the decomposition in \eqref{eq:PhiFisher} and invoking Jensen's inequality as in \cref{res:randomPerturbation}. The proof of the strict inequality is deferred to the Appendix.
\end{proof}

The decrease of the Fisher information demonstrated in \cref{res:thinning} can be quantified in the special case of an independently identically distributed point process as follows.

\begin{proposition}
\label{res:iidThinning}
Let $\Phi$ be an independently identically distributed point process characterised by a weakly-differentiable family of probability distributions parametrised by $\theta \in \Theta$ and assume its cardinality distribution $\pi_{\theta}= \{\pi_{\theta}(n): n \in \bbN_0 \}$ does not depend on $\theta$, then
\eqns{
\big( \bscalI_{\Phi_{\alpha}}(\theta) - \bscalI_{\Phi_{\alpha'}}(\theta) \big) /\bscalI_X(\theta) =
\big((\alpha-\alpha') - (\alpha^2 - \alpha'^2)) \bbE(N) + (\alpha^2 - \alpha'^2) \bbE(N^2) \geq 0
}
for any $0 \leq \alpha' \leq \alpha \leq 1$.
\end{proposition}

\begin{proof}
The parameter $\theta$ of the distribution $\pi_{\theta}$ is omitted in this proof as a consequence of the assumption of independence. Additionally, thinning does not affect the common distribution of the points in $\Phi$ so that, from 
\eqref{eq:iidppFisher}, both point processes have $\mathcal{I}_N(\theta)=0$ and their $\bscalI_{X}(\theta)$ terms are equal. Thus, denoting $N_{\alpha}$ the random number of points in $\Phi_{\alpha}$, the objective is to show that $\bbE(N_{\alpha}^2)$ is greater than $\bbE(N_{\alpha'}^2)$. It holds that the distribution $\pi_{\alpha}$ of $N_{\alpha}$ verifies
\eqns{
\pi_{\alpha}(n) = \sum_{k \geq n} \pi(k) {k\choose n} \alpha^n (1 - \alpha)^{k - n},
}
for any $n \geq 0$, so that
\eqns{
\bbE(N_{\alpha}^2) = \sum_{k \geq 0} \pi(k) \sum_{n = 0}^k n^2  {k\choose n} \alpha^n (1 - \alpha)^{k - n}.
}
The second sum in the right hand side can be recognised to be the second moment of Bernoulli random variable so that
\eqnsa{
\bbE(N_{\alpha}^2) & = \sum_{k \geq n} \pi(k) k \alpha ((k-1)\alpha + 1) \\
& = (\alpha - \alpha^2) \bbE(N) + \alpha^2 \bbE(N^2),
}
from which the result follows.
\end{proof}

\Cref{res:iidThinning} sheds light on the source of the information loss when applying independent thinning to a point process: the quantity $( \bscalI_{\Phi_{\alpha}}(\theta) - \bscalI_{\Phi_{\alpha'}}(\theta) ) /\bscalI_X(\theta)$, which can be seen as a relative loss of Fisher information, is shown to be related to the first and second moments of the random variable associated with the number of points in the process. This is because the operation of thinning applied to the considered type of independently identically distributed point process incurs a loss of information only through the decrease of the number of points.

The focus is now on how information evolves when the points of $\Phi$ are augmented with that of another point process which has a distribution not depending on $\theta$. In the context of multi-object observation models, the point process being augmented to $\Phi$ are spurious observations called clutter which is unrelated to the objects being tracked, e.g.\ generated by radar reflections from non-targets. This, combined with the fact that the number of clutter points received is \emph{a priori} unknown, shows that treating clutter as a $\theta$-independent point process is appropriate. 
Superposition is less straightforward than thinning since the resulting augmented point process will have an altered spatial distribution and cardinality distribution. However, the operation of superposition can be expressed as a Markov kernel that transforms $\Phi$ to a new point process $\Phi'$ and this Markov kernel is independent of $\theta$. Thus the same approach as in \cref{res:thinning} can be applied to show that superposition (in general) also leads to a loss of Fisher information. In the following \lcnamecref{res:superposition}, $\Phi + \tilde\Phi$ stands for the point process resulting from the superposition of $\Phi$ with another point process $\tilde\Phi$.

\begin{proposition}
\label{res:superposition}
Let $\Phi$ be a point process characterised by a weakly-differentiable family of probability distributions parametrised by $\Theta$ and let $\tilde\Phi$ be another point process whose conditional distribution given $\Phi$ does not depend on $\theta$. Then $\bscalI_{\Phi}(\theta) \geq \bscalI_{\Phi + \tilde\Phi}(\theta)$.
\end{proposition}

\begin{proof}
Let $\tilde{P}( \cdot \given \Phi)$ be the conditional law of $\tilde\Phi$ given $\Phi$, then the law of the point process $\Phi + \tilde\Phi$ given a realisation $(x_1,\dots,x_k)$ of $\Phi$ is
\eqnsml{
Q(B_1 \times \dots \times B_n \given x_1\dots,x_k) = \\
\dfrac{1}{n!} \sum_{\sigma \in \Sym(n)} \ind{B_{\sigma(1)}\times\dots\times B_{\sigma(k)}}(x_1,\dots,x_k)
\tilde{P}(B_{\sigma(k+1)}\times\dots\times B_{\sigma(n)} \given x_1,\dots,x_k)
}
for any $B_1,\dots,B_n \in \calB(\bbR^d)$. The desired can be now established by proceeding as in the proof of \cref{res:thinning}; details are omitted.
\end{proof}

\section{Fisher information via the weak derivative}
\label{sec:measTheroreticFisher}

To start with, the derivative $P'_{\theta}$ has to be be defined formally. For this purpose, we consider the following weak form of measure-valued differentiation \citep{Pflug1992}, where the notation $\mu(f)$ is used to denote the integral $\int f(x) \mu(\d x)$. Henceforth, the set $E$ will be assumed to be Polish with $\calB(E)$ its Borel $\sigma$-algebra.

\begin{definition}
\label{def:derivMeasure}
Let $\{\mu_{\theta}\}_{\theta \in \Theta}$ be a parametric family of finite measures on $(E,\calB(E))$, then $\theta \to \mu_{\theta}$ is said to be \emph{weakly differentiable} at $\theta \in \Theta$ if there exists a signed finite measure $\mu'_{\theta}$ on $(E,\calB(E))$ such that
\eqns{
\lim_{\epsilon \to 0} \dfrac{1}{\epsilon}\big( \mu_{\theta+\epsilon}(f) - \mu_{\theta}(f) \big) = \mu'_{\theta}(f)
}
holds for all bounded continuous functions $f$ on $E$.
\end{definition}

Although the signed measure $\mu'_{\theta}$ is only characterised by the mass is gives to bounded continuous functions, one can show that this characterisation is sufficient to define $\mu'_{\theta}$ on the whole Borel $\sigma$-algebra $\calB(E)$, see \cref{res:uniqueMeasure} in the Appendix.

Assuming that $\theta \mapsto P_{\theta}$ has a derivative at $\theta \in \Theta$, that $P_{\theta}'$ is absolutely continuous with respect to $P_{\theta}$, and that the square of the score is integrable, the Fisher information is defined to be
\eqns{
\bscalI(\theta) = \int \bigg(\dfrac{ \d P'_{\theta}}{\d P_{\theta} }(x)\bigg)^2 P_{\theta}(\d x).
}

Simple cases where this more versatile definition of Fisher information is useful can be given using Dirac measures on the real line as in the following examples.

\begin{example}
\label{ex:staticDiscrete}
Consider $\Theta = [0,1]$, $P_{\theta} = \theta \delta_{-x} + (1-\theta) \delta_x$ for some given $x \in E = \bbR$. Indeed, in this case, $P_{\theta}$ is not absolutely continuous with respect to the natural reference measure on the real line, the Lebesgue measure $\lambda$. However, 
\eqns{
P'_{\theta} = \delta_{-x} - \delta_x,
}
which is a signed measure and 
\eqns{
\dfrac{ \d P'_{\theta}}{\d P_{\theta} } = \dfrac{1}{\theta}\ind{\{-x\}} - \dfrac{1}{1-\theta}\ind{\{x\}},
}
where the Radon-Nikodym derivative is assumed without loss of generality to be equal to $0$ everywhere it is not uniquely defined. It follows from basic calculations that
\eqns{
\bscalI(\theta) = \dfrac{1}{\theta(1-\theta)}.
}
\end{example}

This unsurprisingly  is the Fisher information of a Bernoulli experiment with probability of success equal to $\theta$. \Cref{ex:staticDiscrete} is meant to be an illustrative calculation executing the definition of $\bscalI(\theta)$: indeed the same result can be recovered by simply restricting the domain of definition of $P_{\theta}$ to the set $\{-x,x\}$ for all $\theta \in \Theta$. The following result illustrates a usual setting one would expect both definitions of the Fisher information to coincide.

\begin{proposition}
\label{res:consistency}
For some dominating measure $\lambda$, assume $P_{\theta} \ll \lambda$ for all $\theta \in \Theta$ and let $p_{\theta}$ denote its density. For each $x$, assume $p_{\theta}(x)$ is differentiable w.r.t.\ $\theta$ and 
\eqnl{eq:consistency}{
 \Big| \dfrac{\partial}{\partial \theta} p_{\theta}(x) \Big| \leq g(x)
}
for all $\theta \in \Theta$ and $\lambda$-almost all $x \in E$ where $g$ is some integrable function on $E$.  Then $\bscalI(\theta) = \calI(\theta)$.
\end{proposition}
\begin{remark}
The assumption of \cref{eq:consistency} is often invoked in the analysis of maximum likelihood estimation \citep{Douc2004, Dean2014} to interchange the order of integration and differentiation, and thus not unique to us. An alternative to assumption in \cref{eq:consistency} is to assume that the mapping $\theta\rightarrow \int\left|\dfrac{\partial}{\partial \theta}{p}_{\theta}(x)\right|\lambda(\d x)<\infty$
is a continuous function of $\theta$. This will  imply 
\begin{equation}
\lim_{\epsilon\rightarrow0}\int\left|\frac{p_{\theta+\epsilon}-p_{\theta}}{\epsilon}
-\dfrac{\partial}{\partial \theta} p_{\theta}\right|\lambda(\d x)=0 \label{eq:consistency2}
\end{equation}
 and thus preserving the conclusion of \cref{res:consistency}. The proof of \cref{eq:consistency2} follows similarly to that of \cite[lemma 7.6]{VanDerVaart1998}.
\end{remark}
 
\begin{proof}
Recalling that the probability density function $p_{\theta}$ of $P_{\theta}$ with respect to $\lambda$ is defined as
\eqns{
P_{\theta}(A) = \int \ind{A}(x) p_{\theta}(x) \lambda(\d x)
}
for all $A \in \calB(E)$, it follows from Leibniz's rule that
\eqnsa{
P'_{\theta}(f) & = \lim_{\epsilon \to \infty} \dfrac{1}{\epsilon} \int f(x) \big(p_{\theta + \epsilon}(x) - p_{\theta}(x)\big) \lambda(\d x), \\
& = \int f(x) \dfrac{\partial}{\partial \theta} p_{\theta}(x) \lambda(\d x),
}
for any bounded continuous mappings $f$ on $E$, and we conclude that $\frac{\partial}{\partial \theta} p_{\theta}$ is the Radon-Nikodym derivative of $P'_{\theta}$ with respect to $\lambda$. Rewriting the Fisher information $\bscalI(\theta)$ as
\eqns{
\bscalI(\theta) = \int \Bigg(\dfrac{ \frac{\d P'_{\theta}}{\d\lambda}(x)}{p_{\theta} (x)}\Bigg)^2 P_{\theta}(\d x) = \int \bigg(\dfrac{ \frac{\partial}{\partial \theta}p_{\theta}(x)}{p_{\theta}(x)}\bigg)^2 p_{\theta}(x) \lambda(\d x)  = \calI(\theta).
}
concludes the proof of the \lcnamecref{res:consistency}.
\end{proof}

The proposed expression of Fisher information can be easily extended to cases where the parameter $\theta$ is vector-valued: each component of the Fisher information matrix can be simply defined based on the partial version of the weak differentiation introduced in \cref{def:derivMeasure}.

Another Polish space $F$ is now considered in order to study the Fisher information for probability measures on product spaces. A function $Q$ on $E \times \calB(F)$ is said to be a \emph{signed kernel} from $E$ to $F$ if $Q(x,\cdot)$ is a signed finite measure for all $x \in E$ and if $Q(\cdot,B)$ is measurable for all $B \in \calB(F)$ (with $\bbR$ equipped with the Borel $\sigma$-algebra, which will be considered by default). If, in particular, $Q(x,\cdot)$ is a probability measure for all $x\in E$ then $Q$ is said to be a \emph{Markov kernel}. If $P$ is a probability measure on $E$ then we denote by $P\times Q$ the probability measure on $(E\times F, \calB(E) \otimes \calB(F))$ characterised by $P\times Q(A\times B) = \int \ind{A}(x) Q(x,B) P(\d x)$ for all $A\times B$ in the product $\sigma$-algebra $\calB(E) \otimes \calB(F)$. A family $\{Q_{\theta}\}_{\theta \in \Theta}$ of Markov kernels from $E$ to $F$ is said to be weakly-differentiable if the measure $Q_{\theta}(x,\cdot)$ is weakly-differentiable for all $x \in E$ and for all $\theta \in \Theta$; it is additionally said to be \emph{bounded weakly-differentiable} if
\eqns{
\sup_g \bigg| \int g(y) Q'_{\theta}(x,\d y) \bigg| < \infty,
}
where the supremum is taken over all bounded continuous functions. If the latter condition is satisfied, then $Q'_{\theta}$ is itself a signed kernel (see \cite[theorem~1]{Heidergott2008_derivatives}). Some technical results are first required.

A formal approach to the weak differentiability of product measures has been considered in \cite{Heidergott2010} and we consider here an easily-proved corollary of \cite[theorem~6.1]{Heidergott2010}.

\begin{corollary}
\label{res:derivativeProduct}
Let $\{P_{\theta}\}_{\theta\in\Theta}$ be a weakly-differentiable parametric family of probability measures on $E$ and let $\{Q_{\theta}\}_{\theta \in \Theta}$ be a bounded weakly-differentiable parametric family of Markov kernels from $E$ to $F$, then
\eqns{
(P_{\theta} \times Q_{\theta})' = P'_{\theta} \times Q_{\theta} + P_{\theta} \times Q'_{\theta}.
}
\end{corollary}

\Cref{res:derivativeProduct} was used at several occasions in the examples of \cref{sec:practicalExample} for the special case where the kernel does not depend on $\theta$, that is $(P_{\theta} \times Q)' = P'_{\theta} \times Q$. In these examples, the key argument was the simplification of terms that appear both in the numerator and denominator of the score function, using the following \lcnamecref{res:jointRadonNikodym}.

\begin{lemma}
\label{res:jointRadonNikodym}
Let $\mu$ and $\tau$ be finite signed measures on $(E,\calB(E))$ such that $\mu \ll \tau$ and let $\nu$ and $\eta$ be signed kernels from $E$ to $F$ such that $\nu(x,\cdot) \ll \eta(x,\cdot)$ for all $x \in E$, then
\eqns{
\dfrac{\d \mu \times \nu}{\d \mu \times \eta}(x,y) = \dfrac{\d \nu(x,\cdot)}{\d \eta(x,\cdot)}(y), \qquad
\dfrac{\d \tau \times \eta}{\d \mu \times \eta}(x,y) = \dfrac{\d \tau}{\d \mu}(x)
}
for $(\mu\times\eta)$-almost every $(x,y) \in E\times F$.
\end{lemma}

\begin{proof}
Denoting $f$ the Radon-Nikodym derivative of $\mu \times \nu$ by $\mu \times \eta$, it holds by definition that
\eqns{
\mu \times \nu(A \times B) = \int \ind{A\times B}(x,y) f(x,y) \mu \times \eta(\d(x,y))
}
for all $A \times B \in \calB(E)\otimes\calB(F)$, so that
\eqns{
\int \ind{A}(x) \nu(x, B) \mu(\d x) = \int \ind{A}(x) \int \ind{B}(y) f(x,y) \eta(x,\d y) \mu(\d x)
}
which implies that, for all $B \in \calB(F)$, it holds that
\eqnl{eq:proof:jointRadonNikodym}{
\nu(x, B) = \int \ind{B}(y) f(x,y) \eta(x,\d y)
}
for $\mu$-almost every $x \in E$. Since $F$ is a Polish space, there exists a countable collection $\calG$ of subsets of $F$ that is a $\pi$-system and that is generating $\calB(F)$. \Cref{eq:proof:jointRadonNikodym} implies that for all $B \in \calG$, there exists a subset $E_B$ of $E$ with full $\mu$-measure such that $\nu(x, B) = \int \ind{B}(y) f(x,y) \eta(x,\d y)$ is true for all $x \in E_B$. Considering the countable intersection $E_{\calG} = \bigcap_{B \in \calG} E_B$, it follows that the statement of interest is true for all $x \in E_{\calG}$ and all $B \in \calG$. To prove the equality of the measures defined on each side of \cref{eq:proof:jointRadonNikodym} it is sufficient to prove their equality on a $\pi$-system as demonstrated. We conclude that $f(x,\cdot)$ is also the Radon-Nikodym derivative of $\nu(x,\cdot)$ by $\eta(x,\cdot)$ for $\mu$-almost every $x$, which proves the first result. The second result can be proved in a similar but simpler way.
\end{proof}

Now assuming that the interest is in the marginal law $\hat{P}_{\theta}$ of $P_{\theta} \times Q_{\theta}$ on $(F,\calB(F))$, it is often easier to express $\hat{P}_{\theta}$ as
\eqns{
\hat{P}_{\theta}(B) = P_{\theta}Q_{\theta}(B) \defeq \int \ind{B}(y) Q_{\theta}(x, \d y) P_{\theta}(\d x),
}
for any $B \in \calB(F)$. In this case, the score can be computed as in the following \lcnamecref{res:FisherIdentity}.

\begin{proposition}[Fisher identity]
\label{res:FisherIdentity}
Let $\hat{P}_{\theta}$ be the law of a random variable $Y$ from $(\Omega,\Sigma,\bbP)$ to $(F,\calB(F))$ defined as the marginal of the law $P_{\theta} \times Q_{\theta}$ of $(X,Y)$ on $(E \times F, \calB(E) \otimes \calB(F))$, and let $\{P_{\theta}\}_{\theta\in\Theta}$ and $\{Q_{\theta}\}_{\theta\in\Theta}$ be respectively weakly-differentiable and bounded weakly-differentiable, then
\eqnl{eq:FisherIdentity}{
\dfrac{ \d \hat{P}'_{\theta}}{\d \hat{P}_{\theta} }(Y) = \bbE_{\theta}\bigg( \dfrac{ \d (P_{\theta} \times Q_{\theta})'}{\d P_{\theta} \times Q_{\theta} }(X,Y) \Given Y \bigg) \qquad \text{almost surely}
}
with $\bbE_{\theta}(\cdot \given Y)$ the conditional expectation for a given $\theta \in \Theta$.
\end{proposition}

\begin{proof}
For any $\theta \in \Theta$, the marginal $\hat{P}_{\theta}$ is simply the probability measure $B \mapsto P_{\theta} \times Q_{\theta}(E \times B)$, so that the family $\{\hat{P}_{\theta}\}_{\theta \in \Theta}$ inherits weak-differentiability from $\{P_{\theta}\}_{\theta \in \Theta}$ and $\{Q_{\theta}\}_{\theta \in \Theta}$. The derivative $\hat{P}'_{\theta}$ can then be characterised for all $B \in\calB(E)$ by
\eqnsa{
\hat{P}'_{\theta}(B) & = (P_{\theta} \times Q_{\theta})'(E \times B) \\
& = \int \dfrac{ \d (P_{\theta} \times Q_{\theta})'}{\d P_{\theta} \times Q_{\theta} }(x,y) \ind{E\times B}(x,y) P_{\theta} \times Q_{\theta}(\d(x,y)) \\
& = \int \bbE_{\theta}\bigg( \dfrac{ \d (P_{\theta} \times Q_{\theta})'}{\d P_{\theta} \times Q_{\theta} }(X,Y) \Given Y = y \bigg) P_{\theta}Q_{\theta}(\d y)
}
Recalling that $\hat{P}_{\theta} = P_{\theta}Q_{\theta}$ concludes the proof of the \lcnamecref{res:FisherIdentity}.
\end{proof}

The Fisher identity is particularly important when the interest is in the Fisher information with respect to the successive observations of a state space model \citep{Douc2004, Dean2014}, in which case it is defined as the limit
\eqns{
\bscalI(\theta) = \lim_{n\to\infty} \dfrac{1}{n} \int \bigg(\dfrac{ \d \bar{P}'_{\theta}}{\d \bar{P}_{\theta} }(y_1,\dots,y_n)\bigg)^2 \bar{P}_{\theta}(\d(y_1,\dots,y_n)),
}
where $n$ refers to the time horizon and where $\bar{P}_{\theta}$ is the stationary distribution of the observation process.

The results of \cref{res:derivativeProduct,res:jointRadonNikodym} also lead to the following extension of a known property of Fisher information, involving the Fisher information $\bscalI_{Y|X}(\theta)$ of a random variable $Y$ calculated with respect to the conditional law of $Y$ given another random variable $X$, defined as
\eqns{
\bscalI_{Y|X}(\theta) = \int \bscalI_{Y}(\theta; x) P(\d x),
}
where $P$ is the law of $X$ and where
\eqns{
x \mapsto \bscalI_{Y}(\theta; x) = \int \bigg(\dfrac{ \d Q'_{\theta}(x,\cdot)}{\d Q_{\theta}(x,\cdot)}(y)\bigg)^2 Q_{\theta}(x, \d y)
}
is assumed to be a measurable mapping, with $Q_{\theta}$ a Markov kernel identified with the conditional law of $Y$ given $X$. Note that making the law of $X$ dependent on the parameter $\theta$ does not induce any difficulties.

\begin{proposition}
\label{res:additivityInformation}
Let $X$ and $Y$ be random variables on a common probability space $(\Omega,\Sigma,\bbP)$ whose laws are parametrised by $\theta \in \Theta$, let the family of laws of $X$ be weakly-differentiable, and let the family of laws of $Y$ given $X$ be bounded and weakly-differentiable, then the Fisher information $\bscalI_{X,Y}(\theta)$ corresponding to the law of $(X,Y)$ can be expressed as
\eqns{
\bscalI_{X,Y}(\theta) = \bscalI_{Y|X}(\theta) + \bscalI_X(\theta)
}
where $\bscalI_{Y|X}(\theta)$ and $\bscalI_X(\theta)$ correspond to the random variables $Y|X$ and $X$ respectively.
\end{proposition}

\begin{proof}
Let $\{P_{\theta}\}_{\theta\in\Theta}$ be the (weakly-differentiable) parametric family of laws of $X$ and let $\{Q_{\theta}\}_{\theta \in \Theta}$ be the (bounded weakly-differentiable) parametric family of conditional laws of $Y$ given $X$, then
\eqns{
\bscalI_{X,Y}(\theta) = \int \bigg(\dfrac{ \d (P_{\theta} \times Q_{\theta})'}{\d P_{\theta} \times Q_{\theta}}(x,y)\bigg)^2 P_{\theta} \times Q_{\theta}(\d (x,y)).
}
Using \cref{res:derivativeProduct,res:jointRadonNikodym}, it follows that
\eqns{
\bscalI_{X,Y}(\theta) =  \int \bigg(\dfrac{ \d P'_{\theta} }{\d P_{\theta}}(x) + \dfrac{ \d Q'_{\theta}(x,\cdot)}{\d Q_{\theta}(x,\cdot)}(y)\bigg)^2 P_{\theta} \times Q_{\theta}(\d (x,y))
}
which concludes the proof of the \lcnamecref{res:additivityInformation}.
\end{proof}

A straightforward corollary of \cref{res:additivityInformation} can be stated as follows: if $X$ and $Y$ are independent random variables, then $\bscalI_{X,Y}(\theta) = \bscalI_X(\theta) + \bscalI_Y(\theta)$. Note that \cref{res:additivityInformation} could also be used to prove \cref{res:randomPerturbation}.

\subsection*{Acknowledgements}
S.S.\ Singh would like to thank Prof.\ Ioannis Kontoyiannis for helpful remarks. All authors were supported by Singapore Ministry of Education AcRF tier 1 grant R-155-000-182-114. AJ is affiliated with the Risk Management Institute, OR and analytics cluster and the Center for Quantitative Finance at NUS. 

\appendix
\section*{Appendix}
\subsection*{Proofs and technical details}

\begin{proof}[Proof of strict inequality in \cref{res:thinning}]
Jensen's inequality is strict unless it is applied to a non-strictly-convex function or to a degenerate random variable. In the context of \cref{res:thinning}, the involved function is $y \mapsto y^2$ so that we only have to verify that the random variable
\eqns{
S_{\theta}(\Phi) = \frac{\d P'_{\theta}}{\d P_{\theta}}(\Phi)
}
is not $\sigma(\Phi_{\alpha})$-measurable:
\begin{enumerate}[wide]
\item We can rule out $S_{\theta}(\Phi) = c$ (for some constant $c$) almost surely as follows: since $\bbE( S_{\theta}(\Phi)) =0$,
it follows that $c=0.$ But this violates the assumption that $\bscalI_{\Phi}(\theta) > 0$.
\item Since $S_{\theta}(\Phi)$ is not a constant almost surely, there exists a set $A \in \calB(\bbR)$ such that $1>\bbE\big(\mathbb{I}_{A}( S_{\theta}(\Phi)) \big) >0$.
Then 
\eqnla{eq:strict}{
\bbE\big( \bbI_{A}\left(S_{\theta}(\Phi)\right)\bbI_{\bbR^{0}}(\Phi_{\alpha})\big) & = \bbE\big( \bbI_{A}(S_{\theta}(\Phi)) \bbE(\bbI_{\bbR^{0}}(\Phi_{\alpha}) \given \Phi)\big) \\
& = \bbE\big( \bbI_{A}(S_{\theta}(\Phi))(1-\alpha)^{|\Phi|} \big) > 0
}
since $(1-\alpha)^{|\Phi|}>0$ almost surely where $|\Phi|$ denotes the number of points in $\Phi$ and where $\bbI_{\bbR^{0}}(\Phi_{\alpha})$ is the indicator of the event $|\Phi_{\alpha}| = 0$. We can similarly show that \cref{eq:strict} holds with $A$ replaced with $A^{c}$. Thus 
\begin{align*}
\bbE\big( \bbI_{\bbR^{0}}(\Phi_{\alpha})\big)  & >\bbE\big( \bbI_{A}(S_{\theta}(\Phi))\bbI_{\bbR^{0}}(\Phi_{\alpha})\big) > 0
\end{align*}
which violates the following fact: Let $X$ and $Y$ be integrable random variables, assume $Y=c$ is an atom of $\sigma(Y)$ and $Y=c$ has positive probability. If $X$ is $\sigma(Y)$ measurable then $\bbE( \bbI_{A}(X)\bbI_{\{ c\} }(Y) )$ is either $0$ or equal to $\bbE( \bbI_{\{c\}}(Y) )$.
\end{enumerate}
\end{proof}

\begin{lemma}
\label{res:uniqueMeasure}
If $\mu$ be a finite signed measure on a metric space $E$ characterised by the value of $\mu(f)$ for all bounded continuous mappings $f$ on $E$. Then $\mu$ is uniquely defined on $\calB(E)$.
\end{lemma}
\begin{proof}
Let $\tau$ be another finite signed measure that is characterised by $\tau(f) = \mu(f)$ for all bounded continuous functions $f$ on $E$. We first prove that $\mu$ and $\tau$ agree on the closed subsets of $E$. Let $\rho(x,y)$ be the metric on $E$ and let $\rho(x,C)$ denote the usual distance between a point $x$ and set $C$. Let  $f_{\epsilon}$ be the continuous function $f_{\epsilon}(x) = (1-\rho(x,C)/\epsilon)^+$ for some some closed set $C$ and some $\epsilon > 0$ where $g^+$ denotes the positive part of a function $g$. Note that $f_{\epsilon}(x)$ is a continuous function that approximates $\ind{C}(x)$ and
\eqns{
\ind{C}(x) \leq f_{\epsilon}(x) \leq \ind{C^{\epsilon}}(x)
}
with $C^{\epsilon}$ the $\epsilon$-neighbourhood of $C$, so that $\eta(f_{\epsilon})$ tends to $\eta(C)$ when $\epsilon \to 0$ for any finite signed measure $\eta$. It follows from that relation $\tau(f_{\epsilon}) = \mu(f_{\epsilon})$ that $\tau(C) = \mu(C)$. This result can be extended to $\mu = \tau$ as follows. Noticing that the set $\calG = \{B \in \calB(E): \mu(B) = \tau(B)\}$ is a $\lambda$-system that contains the closed sets and that the set of closed sets are themselves a $\pi$-system (which generates $\calB(E)$), we conclude by the $\pi$-$\lambda$ theorem that $\calB(E)$ is contained in $\calG$. Thus $\calB(E) = \calG$ and therefore $\mu(B) = \tau(B)$ for all $B \in \calB(E)$.
\end{proof}

\bibliographystyle{agsm}
\bibliography{Fisher}

\end{document}